\documentclass[12pt]{amsproc}
\newtheorem{theorem}{\sc Theorem}[section]
\newtheorem{lemma}[theorem]{\sc Lemma}

\newtheorem{problem}[theorem]{\sc Problem}

\title[On just-infiniteness of groups and $C^*$-algebras]{On just-infiniteness of locally finite groups and their $C^*$-algebras}
\author{V. Belyaev}
\address{Institute of Mathematics and Mechanics\\ S. Kovalevskaja 16, Ekaterinburg, Russia}
\email{v.v.belyaev@list.ru}
\author{R. Grigorchuk}
\address{Department of Mathematics\\ Texas A\&M University\\ College Station, TX, USA}
\email{grigorch@math.tamu.edu}
\author{P. Shumyatsky}
\address{Department of Mathematics\\ University of Bras\'\i lia\\ 70910 Bras\'\i lia DF
\\ Brazil}
\email{pavel@unb.br}
\thanks{The second author was supported by NSA grant H98230-15-1-0328.
The third author was supported by FAPDF and CNPq}
\dedicatory{For Alex Lubotzky on his 60th birthday}

\begin{document}
\begin{abstract} We give a construction of a family of locally finite residually finite groups with just-infinite $C^*$-algebra. This answers a question from \cite{gri3}. Additionally, we show that residually finite groups of finite exponent are never just-infinite.
\end{abstract}
\maketitle

\section{Introduction} A group is called just-infinite if it is infinite but every proper quotient is finite. Any infinite finitely generated group has just-infinite quotient. Therefore any question about existence of an infinite finitely generated group with certain property which is preserved under homomorphic images can be reduced to a similar question in the class of just-infinite groups. For instance, we do not know if there exists an infinite finitely generated group in which every element represents a value of the $n$th Engel word $[x,y,\dots,y]$. Clearly, if such groups exist then some of them are just-infinite. This observation is one of the numerous that justify the importance of just-infinite groups.

In \cite[Theorem 2]{gri1} the second author showed that the class of just-infinite groups naturally splits into three subclasses: just-infinite groups of branch type, just-infinite groups containing a subgroup of finite index that is a direct product of finitely many copies of the same group and the latter is either simple or hereditary just-infinite group.

A branch group is a group $G$ that has faithful level transitive action by automorphisms on a spherically homogeneous rooted tree $T_{\bar m}$, given by a sequence ${\bar m}=\{m_n\}_{n=0}^\infty$ of positive integers $m_n\geq2$ (called branch index) with the property that for any $n$ the index of rigid stabilizer $rist_G(n)$ in $G$ is finite \cite{gri1}. A group $G$ is hereditary just-infinite if it is infinite residually finite and every subgroup of finite index (including $G$ itself) is just-infinite.

Thus, basically, the study of just-infinite groups can be reduced to study of branch, hereditary just-infinite and infinite simple groups. While the latter two classes in the above list consist of just-infinite groups, branch groups are not necessarily just-infinite (for instance, $Aut\, T_m$ is branch but not just-infinite) but every proper quotient of a branch group is virtually abelian \cite[Theorem 3]{gri1}.

The above trichotomy for just-infinite groups was deduced from the dichotomy presented in \cite{wi2} where the notions of basal subgroup and the structural subgroup lattice were the main tools.

The concept of just-infiniteness can be defined for other objects in mathematics: for algebras of various types (associative, Lie etc). The article \cite{gri3} is devoted to the study of just-infinite (dimensional) $C^*$-algebras. The main result of \cite{gri3} tells us of tricotomy of such $C^*$-algebras. It turns out that the class of just-infinite $C^*$-algebras splits into three subclasses according to the structure of the primitive ideal space $Prim(A)$ which can consist of a single point, to be finite set of cardinality $\geq 2$, or to be infinite countable. The corresponding cases are marked in \cite[Theorem 3.10]{gri3} as cases $(\alpha),(\beta)$ and $(\gamma)$. The class $(\gamma)$ (most complicated of all) consists of residually finite-dimensional algebras while algebras in $(\alpha)$ and $(\beta)$ are not residually finite-dimensional.

All classes $(\alpha),(\beta)$ and $(\gamma)$ are non-empty and while examples of algebras in the classes $(\alpha)$ and $(\beta)$ are straightforward, the construction of algebras in the class $(\gamma)$ requires some work. The examples that are presented in \cite{gri3} are obtained using special type of Bratteli diagrams.

The property of an algebra to be residually finite-dimensional is analogous to that of residual finiteness in groups. In \cite{gri3} a question about existence of algebras in the class $(\gamma)$ that come from groups has been raised. First of all, the interest is focused on the full $C^*$-algebra $C^*(G)$ and on the reduced $C^*$-algebras $C_r^*(G)$ generated by the left (or, what is essentially the same, right) regular representation of a group. But examples of the type $C^*_\rho(G)$ of algebras generated by a unitary representation $\rho$ of $G$ are also interesting. There is belief that the Koopman representation $\pi$ associated with the action of the Grigorchuk group $\mathcal G$ of intermediate growth constructed in \cite{gri4} and studied in \cite{gri5} and many other articles on the boundary of the rooted binary tree generates $C^*$-algebra $C^*_\pi(\mathcal G)$ which belongs to the class $(\gamma)$.

On the other hand, in \cite{gri3} an attention is paid to the case of locally finite groups and their $C^*$-algebras (in this case $C^*(G)=C^*_r(G)$ because of amenability of $G$). The question 6.10 in \cite{gri3} is as follows.

Does there exist an infinite, residually finite, locally finite group $G$ such that $C^*(G)$ is just-infinite?

The results obtained in \cite{gri3} show that such a group has to be just-infinite and moreover hereditary just-infinite as for branch groups $C^*(G)$ is never just-infinite \cite[Theorem 7.10]{gri3}.

The goal of the present paper is to answer positively the above question and to prove some other related results. In the next section we construct uncountably many residually finite, locally finite groups $G$ for which the algebra $C^*(G)$ is just-infinite (cf Theorem \ref{gla}).

Recall that a group is said to be of finite exponent $n$ if all of its cyclic subgroups have finite order dividing $n$. It is easy to see that the groups $G$ constructed in Section 2 are all of infinite exponent. In Section 3 we show that residually finite groups of finite exponent are never just-infinite. Quite possibly, this is a part of a more general phenomenon. In particular, the following problem seems natural.

\begin{problem} Does there exist an infinite, residually finite, locally finite group $G$ satisfying a nontrivial identity such that $C^*(G)$ is just-infinite?
\end{problem}

\section{Construction of groups $G$ with $C^*(G)$ just-infinite}
Throughout the article our notation is standard. If $G$ is a group and $x,y\in G$, we write $[x,y]$ for $x^{-1}y^{-1}xy$ and $x^y=y^{-1}xy$. Suppose $a$ is an automorphism of $G$. We write $x^a$ for the image of $x$ under $a$.  If $H\leq G$, the subgroup of $G$ generated by elements of the form $h^{-a}h$ with $h\in H$ is denoted by $[H,a]$. It is well-known that the subgroup $[G,a]$ is always an $a$-invariant normal subgroup in $G$. We also write $C_G(a)$ for the fixed point subgroup of the automorphism $a$ in $G$. When convenient, we identify elements of the group $G$ with the corresponding inner automorphisms of $G$.

 We will require the following well-known lemmas.
\begin{lemma}\label{zuzu} Let $G$ be a metabelian group and $a\in G$. Suppose that $B$ is an abelian subgroup in $G$ containing the commutator subgroup $G'$. Then $[B,a]$ is normal in $G$.
\end{lemma}
\begin{proof} Since $B$ contains $G'$, it follows that $B$ is normal in $G$. Let $x\in G$. Write $a^x=ay$ with $y\in G'$. Taking into account that $[B,y]=1$, we obtain $[B,a]^x=[B,a^x]=[B,ay]=[B,a]$. Thus, an arbitrary element $x\in G$ normalizes $[B,a]$, as required.
\end{proof}

Suppose that a finite group $H$ is a Cartesian product of nonabelian simple groups $H_i$. It is well-known that any normal subgroup of $H$ can be written as a product of some of the factors $H_i$. Therefore any automorphism of $H$ permutes the simple factors. These facts are used in the following lemma.
\begin{lemma}\label{zok} Let $H=H_1\times\dots\times H_s$ be a Cartesian product of nonabelian simple groups $H_i$ and let $a$ be an automorphism of $H$. Suppose that none of the simple factors $H_i$ lies in $C_H(a)$. Then $H=[H,a]$.
\end{lemma}
\begin{proof} By induction on $s$ we can assume that $a$ transitively permutes the simple factors  $H_i$. Suppose that the lemma is false and $[H,a]\neq H$. Since any normal subgroup of $H$ can be written as a product of some of the factors $H_i$, we can assume that $[H,a]=H_1\times\dots\times H_t$ for some $t<s$. Then it is clear that the subgroup $H_0=H_{t+1}\times\dots\times H_s$ is $a$-invariant and so $H=[H,a]\times H_0$ with both subgroups $[H,a]$ and $H_0$ being $a$-invariant. However, this contradicts the assumption that the automorphism $a$ transitively permutes the factors $H_i$.
\end{proof}

We write $X\wr Y$ for the wreath product with the active group $Y$ and the passive group $X$. Thus, the base subgroup $B$ of $X\wr Y$ is the direct product of isomorphic copies of the group $X$. If $X$ and $Y$ are finite groups with $X$ being simple, then $B$ is a unique minimal normal subgroup in $X\wr Y$. The minimality follows immediately from the fact that $Y$ transitively permutes the simple factors of $B$. The uniqueness can be deduced by observing that any other minimal normal subgroup of $X\wr Y$ centralizes $B$ while the definition of $X\wr Y$ implies that the centralizer of $B$ is trivial.

Let $S_1,S_2,\dots$ be a sequence of nonabelian finite simple groups. Here we do not assume that the groups $S_i$ are pairwise non-isomorphic. In fact, they can be all isomorphic.

Set $G_1=S_1$ and $G_{i+1}=S_{i+1}\wr G_i$ for $i=1,2,\dots$. In the natural way, $G_i$ will be identified with the active group of $G_{i+1}$. Hence, for each $i$ there is a natural embedding of $G_i$ into $G_{i+1}$. Denote by $M_i$ the unique minimal normal subgroup of $G_{i+1}$. Thus, $M_i$ is a direct product of isomorphic copies of $S_{i+1}$ and $G_i$ regularly permutes the simple factors.
\begin{lemma}\label{lala} For each $i$ we have
 \begin{itemize} \item[(a)] The group $G_{i+1}$ has precisely $i$ proper normal nontrivial subgroups. These are the subgroups of the form $\prod_{j=k}^iM_j$ with $1\leq k\leq i$. 
\item[(b)] Each normal subgroup $\prod_{j=k}^iM_j$ is complemented in  $G_{i+1}$ by the subgroup $G_k$.
\item[(c)] If $N$ is a normal subgroup of $G_{i+1}$ and $x\in G_{i+1}\setminus N$, then $[N,x]=N$.

\end{itemize}
\end{lemma}
\begin{proof} Both claims (a) and (b) can be obtained by induction on $i$, the case where $i\leq 1$ being obvious. Suppose that $i\geq 2$ and let $N$ be a proper normal nontrivial subgroup in $G_{i+1}$. Since $M_i$ is the unique minimal normal subgroup, it follows that $M_i\leq N$. On the other hand, $G_{i+1}/M_i\cong G_{i}$ and by induction the claims (a) and (b) follow. Let us prove Part (c). Let $N$ be a normal subgroup in $G_{i+1}$ and choose $x\not\in N$. By Part (a) we can write $N=\prod_{j=k}^iM_j$ for some $k\leq i$. We use the backward induction on $k$. If $k=i$ and $N=M_i$, the result is straightforward from Lemma \ref{zok}. Suppose that $k\leq i-1$. We already know that $[M_i,x]=M_i$. Passing to the quotient $G_{i+1}/M_i$ and using induction we conclude that $N=[N,x]M_i$. Taking into account that $[M_i,x]=M_i$ we deduce that $[N,x]=N$, as required.
\end{proof}
Let $G$ be the union (direct limit) of the infinite sequence $$G_1\leq G_2\leq\dots.$$ 
For each $i$ we set $T_i=\langle M_j;j\geq i\rangle$, the subgroup generated by all $M_j$ with $j\geq i$. It is clear that the subgroups $T_i$ are normal in $G$. Moreover it follows from Lemma \ref{lala} that for any $i$ the group $G$ is a semidirect product of $T_i$ and $G_i$. Further, all subgroups $T_i$ are normal of finite index and for each normal subgroup $N$ in $G$ there exists $i$ such that $N=T_i$. Indeed, let $N$ be a normal subgroup of $G$ and let $i$ be the minimal index such that $N\cap G_{i+1}\neq1$. Since $M_i$ is the unique minimal normal subgroup in $G_{i+1}$, it follows that $N\cap G_{i+1}=M_i$. Lemma \ref{lala} shows that whenever $s\geq i+1$ the group $G_s$ has precisely one normal subgroup containing $M_i$ and not containing $M_{i-1}$. This is $\prod_{j=i}^{s-1}M_j$. So the intersection $N\cap G_s$ is precisely the product $\prod_{j=i}^{s-1}M_j$. This of course implies that $N=T_i$.

It follows that $G$ is a locally finite residually finite just-infinite group. Let us show that $G$ is a hereditary just-infinite group. Choose a subgroup $H$ of finite index in $G$. Let $N$ be a nontrivial normal subgroup of $H$. The subgroup $H$ contains a subgroup $T_i$ for some $i\geq1$ and it is clear that if $N\cap T_i$ has finite index in $T_i$, then $N$ has finite index in $H$. Hence, without loss of generality we can assume that $H=T_i$. Write $$E_1=M_i,E_2=M_{i+1}M_i,\dots,E_s=\prod_{j=i}^{i+s-1} M_j,\dots.$$ The subgroups $E_s$ are finite and we have inclusions 
$$E_1\leq E_2\leq\dots\leq\cup_{s=1}^\infty E_s=T_i.$$ Let now $s$ be the minimal index for which $N\cap E_s\neq1$ and choose a nontrivial element $x\in N\cap E_s$. Since $x\in E_s$, it follows that $x\not\in M_j$ whenever $j\geq i+s$. Notice that $M_j$ is normal in $G_{j+1}$ and $x\in G_{j+1}$. We apply Lemma \ref{lala} (c) and conclude that $M_j=[M_j,x]$ for every $j\geq i+s$. Recall that $M_j$ normalizes the subgroup $N$ and $x\in N$. It follows that $M_j\leq N$ whenever $j\geq i+s$. Therefore $T_{i+s}\leq N$. Since $T_{i+s}$ has finite index in $G$, so does $N$.

\begin{theorem}\label{kg} Let $K$ be a field. Then every ideal $I$ in the group algebra $K[G]$ has finite codimension.
\end{theorem}
\begin{proof} Let $p$ be the characteristic of the field $K$ (possibly, $p=0$). Choose a prime $q$ such that $q\neq p$ and among the groups $S_i$ there are infinitely many of order divisible by $q$. Such a prime $q$ exists for any sequence of simple groups $\{S_i\}$. Indeed, if $p\neq2$, we can take $q=2$ since by the Feit-Thompson theorem \cite{fetho} the order of each group in the sequence $\{S_i\}$ is even. Suppose $p=2$. If infinitely many of the groups $S_i$ belong to the family of the Suzuki groups, we take $q=5$ since the orders of the Suzuki groups are divisible by 5. If only finitely many of the groups $S_i$ are Suzuki groups, we take $q=3$ since the Suzuki groups are the only simple groups whose orders are not divisible by $3$ \cite{simple}.

Let now $I$ be a nonzero ideal in $K[G]$. We will show that $I$ has finite codimension. Our proof will mimic some arguments from \cite{hartley}.

Let $0\neq\alpha\in I$. Then we can assume that 1 occurs in the support
of $\alpha$ and we write $\alpha=\sum_0^n k_ix_i^{-1}$ with $1=x_0,x_1,\dots,x_n$ distinct elements of $G$ and $k_0\neq0$. Let $u$ be the minimal index such that $x_0,x_1,\dots,x_n\in G_u$. Set $X=\langle x_0,x_1,\dots,x_n\rangle$ and remark that $X$ regularly permutes the simple factors in $M_t$ for any $t>u$. Choose $t$ such that $q\in\pi(M_t)$ and choose an element $y$ of order $q$ that belongs to a simple factor of $M_t$. For $i=0,1,\dots,n$ set $y_i=y^{x_i}$ and $z_i=[y_i,x_i]$. Denote the subgroup $\langle y_i^{x_j};0\leq i,j\leq n\rangle$ by $A$.

We show now by inverse induction on $s$ with $n\geq s\geq0$ that $I$ contains an element $$\beta_s=\sum_0^s\beta_{si}x_i^{-1}$$ such that $\beta_{si}\in K[A]$ and, if $s<n$, we have $$\beta_{s0}=k_0(z_n-1)(z_{n-1}-1)\dots(z_{s+1}-1).$$ First for $s=n$ we merely take $\beta_s=\alpha$. Now suppose we have $\beta_s$ as above contained in $I$ with $n>s>0$. Then $y_s^{-1}\beta_sy_s$ and $z_s\beta_s$ both belong to $I$ and hence
$$\beta_{s-1}=z_s\beta_s-y_s^{-1}\beta_sy_s\in I.$$
Since $A$ is abelian and $y_s,z_s\in A$, we have
$$\beta_{s-1}=\sum_0^s z_s\beta_{si}x_i^{-1}-\sum_0^sy_s^{-1}\beta_{si}x_i^{-1}y_s$$ $$=\sum_0^s\beta_{si}(z_s-[y_s,x_i])x_i^{-1}=\sum_0^{s-1}\beta_{si}(z_s-[y_s,x_i])x_i^{-1}$$ since $z_s=[y_s,x_s]$. Recall that $x_0=1$. It follows that $\beta_{s-1,0}$ has the required form and the induction step is proved.

In particular, when $s=0$ we conclude that $$\beta_0=k_0(z_n-1)(z_{n-1}-1)\dots(z_1-1)\in I.$$ Here $k_0\neq0$ and, since $\langle z_1,z_2,\dots,z_n\rangle=\langle z_1\rangle\times\langle z_2\rangle\times\dots\times\langle z_n\rangle$, we conclude that $\beta_0\neq0$. Thus, we have shown that $I\cap K[A]\neq0$.

Now write $A=\{1=a_0,a_1,\dots,a_m\}$. The group $A$ regularly permutes the simple factors in $M_v$ for any $v\geq t+1$. Choose $v$ such that $q\in\pi(M_v)$ and choose an element $b$ of order $q$ that belongs to a simple factor of $M_v$. For $i=0,1,\dots,m$ set $b_i=b^{a_i}$. Denote the subgroup $\langle b_0,b_1,\dots,b_m\rangle$ by $B$. Set $H=BA$. It is clear that $H\cong C_q\wr A$, where $C_q$ denotes the cyclic group of order $q$. In particular, the center $Z(H)$ is cyclic generated by the product $b_0\cdots b_m$.

Fix $i\leq m$ and set $B_i=[B,a_i]$. Since $B$ is abelian, $B_i$ is isomorphic with $B/C_B(a_i)$. Let $\hat B_i$ denote the sum of elements in $B_i$. We compute $$\sum_{b\in B}(a_i^{-1})^b=\sum_{b\in B}[b,a_i]a_i^{-1}=|C_B(a_i)|\hat B_ia_i^{-1}.$$

Let $\alpha_0=\sum l_ia_i^{-1}\in I\cap K[A]$ with $l_0\neq0$. Then $\beta=\sum_{b\in B}b^{-1}\alpha_0b\in I$. By the above $$\beta=\sum_il_i|C_B(a_i)|\hat B_ia_i^{-1}.$$ For $i=0$, since $a_0=1$, the summand here is $l_0|B|$. For $i\neq0$ the element $a_i$ does not centralize $B$ so $B_i\neq1$. Lemma \ref{zuzu} tells us that each subgroup $B_i$ is normal in $H$. Recall that $Z(H)$ is cyclic of order $q$. Because $H$ is nilpotent, every normal subgroup of $H$ has nontrivial intersection with $Z(H)$. Consequently, $Z(H)\leq B_i$. Let $z$ be a generator of $Z(H)$. Since $z\in B_i$, we have $(z-1)\hat B_i=0$. Therefore $$(z-1)\beta=l_0|B|(z-1).$$ We put this together with the facts that $(z-1)\beta\in I$ and $l_0\neq0$. Since the characteristic of $K$ is not $q$, it follows that $z-1\in I$.

Now set $N=\{g\in G; g-1\in I\}$. This is a normal subgroup in $G$. The above shows that $N\neq1$. Since $G$ is just-infinite, we conclude that $N$ has finite index. This implies that $I$ has finite codimension.
\end{proof}

Remark. The proof of Theorem \ref{kg} uses the classification of finite simple groups in the part where the prime $q$, different from the characteristic of $K$, is chosen. We emphasize however that in the important case where $K$ has characteristic 0 (for instance, $K=\Bbb C$) the classification is not required. Further, the classification is not required for many choices of the sequence $S_1,S_2,\dots$, in particular for the case where all simple groups $S_i$ are isomorphic.
\medskip

The main result of the present article can now be easily obtained by combining Theorem \ref{kg} with results from \cite{gri3}. We denote by
$\Bbb C$ the field of complex numbers.
\begin{theorem}\label{gla} There exists an infinite, residually finite, locally finite group $G$ such that $C^*(G)$ is just-infinite
\end{theorem}
\begin{proof} Let $G$ be as above. By Theorem \ref{kg}, the group algebra $\Bbb C[G]$ is just-infinite. The result now follows immediately from Propositions 6.2 and 6.6 of \cite{gri3}.
\end{proof}

\section{On groups of finite exponent}

In the present section we will show that residually finite groups of finite exponent are never just-infinite.

\begin{lemma}\label{pp} Let $p$ be a prime and $G$ a locally finite $p$-group such that $G/G'$ is finite. Assume that $G$ is residually finite. Then $G$ is  finite.
\end{lemma}
\begin{proof} Since $G/G'$ is finite, we can choose finitely many elements $a_1,\dots,a_k$ such that $G=\langle a_1,\dots,a_k\rangle G'$. Set $H=\langle a_1,\dots,a_k\rangle$. Since in any finite quotient $Q$ of $G$ the derived group $Q'$ is contained in the Frattini subgroup $\Phi(Q)$, we conclude that $Q$ is isomorphic to a finite quotient of $H$. The lemma now follows from the fact that $H$ is finite.
\end{proof}

\begin{lemma}\label{ttt} Let $n$ be a positive integer and $G$ an infinite group that is residually of order at most $n$. Then $G$ has a proper infinite quotient.
\end{lemma}
\begin{proof} Suppose that the lemma is false and $G$ is just-infinite. Denote by $\mathcal F$ the family of all normal subgroups of $G$ having index at most $n$. Since $G$ is infinite, we can choose $n+1$ pairwise distinct elements $g_1,\dots,g_{n+1}\in G$. For all $i,j\leq n+1$ denote by ${\mathcal F}_{ij}$ the family of all subgroups $N\in\mathcal F$ such that $g_ig_j^{-1}\in N$. We see that $\mathcal F=\bigcup_{i\neq j}\mathcal F_{ij}$. Indeed, let $N\in\mathcal F$. Since the index of $N$ in $G$ is at most $n$, it follows that $Ng_i=Ng_j$ for some $i\neq j$. This means that $N\in {\mathcal F}_{ij}$. We further remark that all families ${\mathcal F}_{ij}$ are non-empty since $G\in{\mathcal F}_{ij}$ for any $i,j$.

For each pair of indexes $i\neq j$ denote by $N_{ij}$ the intersection of all subgroups in ${\mathcal F}_{ij}$. Then $N_{ij}\neq1$ because $g_ig_j^{-1}\in N_{ij}$. Taking into account that $G$ is just-infinite we deduce that all $N_{ij}$ have finite index. Obviously, $$\bigcap_{i\neq j}N_{ij}=\bigcap_{N\in\mathcal F}N=1$$ and we conclude that $G$ is finite. This is a contradiction.
\end{proof}

We are now ready to establish the main result of this section.

\begin{theorem}\label{expo} Every infinite residually finite group of finite exponent has a proper infinite quotient.
\end{theorem}
\begin{proof} Let $G$ be an infinite residually finite group of finite exponent and assume that $G$ is just-infinite. Recall that by the solution of the restricted Burnside problem $G$ is locally finite \cite{ze1,ze2}. Let $\hat G$ be the profinite completion of $G$. We view $G$ as a subgroup of $\hat G$. By Wilson's theorem \cite{wilson} $\hat G$ possesses a normal series all of whose factors are either $p$-groups or Cartesian products of isomorphic non-abelian finite simple groups.

Since $G$ is just-infinite, the first term of the series that intersects $G$ nontrivially must intersect it by a subgroup of finite index, say $K$. Suppose first that $K$ is a $p$-subgroup. Since $K$ is infinite, Lemma \ref{pp} implies that $K'$ has infinite index in $G$. Therefore $K'=1$ and $K$ is abelian. But then it is clear that every element of $K$ is contained in a finite normal subgroup of $G$. This contradicts the hypothesis that $G$ is just-infinite. Therefore we will assume that $K$ embeds into a Cartesian product of isomorphic finite simple groups. However now it follows that there exists a positive integer $n$ such that $G$ is residually of order at most $n$. By Lemma \ref{ttt}, $G$ has a proper infinite quotient.
\end{proof}

\end{document}